\documentclass[12pt, reqno]{amsart}
\usepackage{amsmath}  \hoffset=-56pt

\usepackage{amsmath, amsfonts, rawfonts,amssymb,txfonts, enumerate}
\usepackage{eucal}
\usepackage{latexsym}
\usepackage{cite, color}
 \usepackage{bm}

\newcommand{\rnnn}{\mathbb R^n}

\newcommand{\sn}{ {\mathbb{S}^{n-1}}}
\newcommand{\R}{\mathbb R}

\newcommand{\psum}{{+_{\negthinspace\kern-2pt p}}\,}
\newcommand{\qsum}[1]{{+_{\negthinspace\kern-2pt #1}}\,}
\newcommand{\dpsum}{{\tilde+_{\negthinspace\kern-1pt p}}\,}
\newcommand{\dqsum}[1]{{\tilde+_{\negthinspace\kern-1pt #1}}\,}
\newcommand{\lsub}[1]{\hskip -1.5pt\lower.5ex\hbox{$_{#1}$}}

\numberwithin{equation}{section}

\newtheorem{theo}{Theorem}[section]

\newtheorem{lem}[theo]{Lemma}

 \theoremstyle{definition}

\begin{document}

\title{A flow method to the Orlicz-Aleksandrov problem}

\author[J. Hu]{Jinrong Hu}
\address{School of Mathematics, Hunan University, Changsha, 410082, Hunan Province, China}
\email{hujinrong@hnu.edu.cn}
\author[J. Liu]{Jiaqian Liu}
\address{School of Mathematics, Hunan University, Changsha, 410082, Hunan Province, China}
\email{liujiaqian@hnu.edu.cn}
\author[D. Ma]{Di Ma}
\address{School of Mathematics, Hunan University, Changsha, 410082, Hunan Province, China}
\email{madi@hnu.edu.cn}

\begin{abstract}

In this paper, we obtain an existence result of smooth solutions to the Orlicz-Aleksandrov problem from the perspective of geometric flow. Furthermore, a special uniqueness result of solutions to this problem shall be discussed.

\end{abstract}
\thanks{The research is supported, in part, by the National Science Foundation of China (12171144)}
\maketitle

\baselineskip18pt

\parskip3pt

\section{Introduction}

The Brunn-Minkowski theory,  the core content of convex geometry, was developed by Minkowski, Aleksandrov, Fenchel, and many others. The Minkowski problem, involving surface area measure is a classical problem in the Brunn-Minkowski theory, which was proposed and solved by Minkowski himself \cite{M897,M903}. As an important counterpart of the Minkowski problem, the classical Aleksandrov problem was introduced by Aleksandrov in \cite{A42}. It is a characterization problem for the Aleksandrov integral curvature, which was originally solved by Aleksandrov via a topological argument, see \cite{A42}. Different from Aleksandrov's method, Oliker \cite{O07} gave a new proof based on a variational technique inspired by optimal transportation theory, and recently Bertrand \cite{Be16} also provided an alternative approach to treat this problem using optimal mass transport.

The $L_{p}$ Brunn-Minkowski theory, as an analogue of the Brunn-Minkowski theory,  which was initiated by Firey \cite{F67} in the 1950s, but this theory gained a new life when Lutwak enriched it with the concept of the $L_{p}$ surface area measure in \cite{L93} in the early 1990s. The $L_{p}$ Minkowski problem involving the $L_{p}$ surface area measure is a fundamental problem in this theory, which has been the breeding ground equipped with many different results in a series of papers \cite{B17,B19,CW06,HY15,L04,L93,Zhu15}. With the development of the $L_{p}$ Brunn-Minkowski theory, as a parallelism of the $L_{p}$ Minkowski problem, Huang-LYZ\cite{HLYZ18} introduced the $L_{p}$ Aleksandrov problem for the $L_{p}$ Aleksandrov integral curvature, and completely established the existence result for $p>0$. At the same time, Zhao\cite{Zh19} presented a solution to the $L_{p}$ Aleksandrov problem for origin-symmetric polytopes in the case of $-1<p<0$.

The Orlicz-Brunn-Minkowski theory, as a new generation of the $L_{p}$ Brunn-Minkowski theory, was launched by Lutwak, Yang and Zhang\cite{LYZ10,LYZ100}. Based on their work, Haberl-LYZ\cite{HLYZ10} first proposed and solved the Orlicz Minkowski problem. Since then, the relevant Orlicz-Minkowski type problems have been widely studied, see for instance \cite{Huang12,Hong18,G19,Wu19}. Naturally, apart from the Orlicz Minkowski problem, the corresponding measure characterization problem for the Aleksandrov problem has been sought for in the Orlicz-Brunn-Minkowski theory. In the recent excellent work, Feng-He \cite{Fe21} defined the Orlicz integral curvature and posed the Orlicz Aleksandrov problem finding the conditions of a given finite Borel measure $\mu$ on $\sn$ as a multiple of the Orlicz integral curvature of a convex body in $\rnnn$, they also gave the necessary and sufficient conditions for the existence of even solutions to the Orlicz Aleksandrov problem via a variational argument.

It is worth mentioning that for the special case in which the given measure $\mu$ has a positive density $f$ on $\sn$, the existence of the solution of the Orlicz Aleksandrov problem amounts to solving the following {M}onge-{A}mp\`ere type equation,
\begin{equation}\label{Or-Mong}
\gamma \varphi(1/h)h(|\nabla h|^{2}+h^{2})^{-\frac{n}{2}}det(\nabla^{2}h+hI)=f
\end{equation}
for some positive constant $\gamma>0$. It is clear to see that, in the case $\varphi(s)=s^{p}$ with $p\in \R$, the Orlicz Aleksandrov problem reduces to the $L_{p}$ Aleksandrov problem, and the constant $\gamma$ can be merged to $h$ due to the homogeneousness of the $L_{p}$ Aleksandrov integral curvature.  In particular,  when equation \eqref{Or-Mong} corresponds to the regularity of solutions to the Aleksandrov problem, this topic has been intensively studied. For example, Oliker\cite{O83} and Pogorelov\cite{P73} independently investigated it in the circumstance that $f$ is a smooth positive function, and Guan-Li\cite{G97} further dealt with certain degenerate {M}onge-{A}mp\`ere type equations in the setting that $f$ is smooth but only nonnegative. Note that, for general $\varphi$,  the uniqueness of solutions to equation \eqref{Or-Mong} is open.

The main purpose of this paper concerns the existence of the solution of the Orlicz Aleksandrov problem associated with the solvability of equation \eqref{Or-Mong} from the perspective of geometric flow. The geometric flow generated by Gauss curvature was first studied by Firey\cite{F74}. From then on, there appeared various Gauss curvature flows, see for instance, Andrews-Guan\cite{An16}, Chen-Huang-Zhao\cite{CHZ19}, Chou-Wang\cite{CW00}, Li-Sheng-Wang\cite{LSW20,LSW200} and their references therein. In the spirit of their work, we consider a family of strictly convex hypersurfaces $\partial \Omega_{t}$ parameterized by smooth map $X(\cdot,t):\sn\rightarrow \rnnn$ satisfying the following flow equation,
\begin{equation}\label{mainflow}
\left\{
\begin{array}{lr}
\frac{\partial X(x,t)}{\partial t}=-\frac{f(\nu)|X|^{n}\kappa \nu \eta(t)}{\varphi(1/(X\cdot \nu)) }+X(x,t);  \\
X(x,0)=X_{0}(x),
\end{array}\right.
\end{equation}
where $\kappa$ is the Gauss curvature of the hypersurface $\partial \Omega_{t}$ at $X(\cdot,t)$, $\nu=x$ is the unit outer normal vector of $\partial\Omega_{t}$ at $X(\cdot,t)$, and $\eta(t)$ is defined by
\begin{equation*}
\eta(t)=\frac{\int_{\sn}du}{\int_{\sn}\frac{f(x)}{\varphi(1/(X\cdot \nu))}dx}.
\end{equation*}

With the aid of the flow equation \eqref{mainflow}, we are devoted to solving equation \eqref{Or-Mong}. Before elaborating the main results of this paper, we do some preparation  with setting
\[
\Phi_{0}=\left\{\Re:\lim_{s\rightarrow 0}\Re(s)=0,\lim_{s\rightarrow \infty}\Re(s)=\infty\right\},
\]
where $\Re:(0,\infty)\rightarrow (0,\infty)$ is continuously differentiable and strictly increasing function. On the other hand, we set
\[
\Psi_{0}=\left\{\Re:\lim_{s\rightarrow 0}\Re(s)=\infty,\lim_{s\rightarrow \infty}\Re(s)=0\right\},
\]
where $\Re:(0,\infty)\rightarrow (0,\infty)$ is continuously differentiable and strictly decreasing function.

We are now in a position to state that the main aim of current work is to obtain the long time existence and convergence results of the flow \eqref{mainflow}. It is shown in the following theorem.
\begin{theo}\label{main*}
Let $\Omega_{0}$ be a smooth, origin symmetric and strictly convex body in $\rnnn$. Suppose $f:\sn\rightarrow (0,\infty)$ is smooth and even, and $\varphi:(0,\infty)\rightarrow (0,\infty)$ is smooth. If there is
\begin{enumerate}
\item[(i)]
$\Re \in \Phi_{0}$ satisfying
$
\Re(s)=\int_{0}^{s}\frac{1}{t \varphi(t)}dt \ with \ \Re^{'}\neq 0 \ on \ (0,\infty)$, for some $\hat{C}>0$, we have
\[
\int_{\sn}f(x)\Re (1/|x\cdot \theta|)dx\leq \hat{C}, \ \forall \theta\in \sn,
\]
and choosing a $\Omega_{0}$ such that
\[
\int_{\sn}f(x)\Re(1/h(\Omega_{0},x))dx> \hat{C};
\]
or
\item[(ii)]
 $\Re \in \Psi_{0}$ satisfying
$
\Re(s)=\int_{s}^{\infty}\frac{1}{t \varphi(t)}dt  \ with \ \Re^{'}\neq 0 \ on \ (0,\infty),
$
\end{enumerate}
 then there exists a smooth, origin symmetric, strictly convex solution $\Omega_{t}$ to flow equation \eqref{mainflow} for all time $t>0$, and it converges along a sequence in $C^{\infty}$ to a smooth, origin symmetric, strictly convex solution to equation \eqref{Or-Mong} for some positive constant $\gamma$.

\end{theo}

 In view of Theorem \ref{main*}, it should be remarked that, since $\Re$ is decreasing in condition (ii), we automatically have
\[
\int_{\sn}f(x)\Re (1/|x\cdot \theta|)dx<\infty, \ \forall \theta\in \sn.
\]
Hence, the integral conditions in (i) compared to the condition (ii) may be not more restrictive.

The organization of this paper goes as follows: In Section \ref{Sec2}, we collect some basic knowledge about convex bodies. In Section \ref{Sec3}, we introduce the geometric flow and functional associated with the Orlicz Aleksandrov problem. In Section \ref{Sec4} and \ref{Sec5}, we obtain the priori estimates of the solution to the relevant flow. In Section \ref{Sec6}, we complete the proof of Theorem \ref{main*}. At last, we shall provide a special uniqueness result to the solution of equation \eqref{Or-Mong} under mild monotonicity assumption on $\varphi$.

\section{Preliminaries}
\label{Sec2}
In this section, we list some basic facts regarding convex bodies. For  quick and good references, please refer to the books of Gardner \cite{G06} and Schneider \cite{S14}.

Our setting will put on the $n$-Euclidean space ${\rnnn}$. Denote by ${\sn}$ the unit sphere. A convex body is a compact convex set of ${\rnnn}$ with non-empty interior. For $Y,Z\in {\rnnn}$, $Y\cdot Z$ stands for the standard inner product. For $Y\in{\rnnn}$, we denote by $|Y|=\sqrt{Y\cdot Y}$ the Euclidean norm.

Let $\Omega$ be a convex body containing the origin in ${\rnnn}$, and let $h(\Omega,\cdot)$ be the support function of $\Omega$ (with respect to the origin), i.e.,  for any $x\in {\sn}$,
\begin{equation*}
h(\Omega,x)=\max\{x\cdot Y:Y \in \Omega \}.
\end{equation*}
The map $g:\partial \Omega\rightarrow {\sn}$ denotes the Gauss map of $\partial\Omega$. Meanwhile, for $\omega\subset {\sn}$, the inverse of Gauss map $g$ is expressed as
\begin{equation*}
g^{-1}(\omega)=\{Z\in \partial \Omega:  g(Z) {\rm \ is \ defined \ and }\ g(Z)\in \omega\}.
\end{equation*}
For simplicity in the subsequence, we abbreviate $g^{-1}$ as $X$.
For a convex body $\Omega$ being of class $C^{2}_{+}$, i.e., its boundary is of class $C^{2}$ and of positive Gauss curvature, the support function of $\Omega$ can be written as
\begin{equation}\label{hhom}
h(\Omega,x)=x\cdot X(x)=g(Z)\cdot Z, \ {\rm where} \ x\in {\sn}, \ g(Z)=x \ {\rm and} \ Z\in \partial \Omega.
\end{equation}
In fact, we can parametrize $\partial \Omega$ by $X(x)$. Let $\{e_{1},e_{2},\ldots, e_{n-1}\}$ be a local orthonormal frame on ${\sn}$, $h_{i}$ be the first order covariant derivatives of $h(\Omega,\cdot)$ on ${\sn}$ with respect to a local orthonormal frame. Differentiating \eqref{hhom} with respect to $e_{i}$ , we get
\[
h_{i}=e_{i}\cdot X(x)+x\cdot X_{i}(x).
\]
Since $X_{i}$ is tangent to $ \partial \Omega$ at $X(x)$, we obtain
\begin{equation}\label{Fi}
h_{i}=e_{i}\cdot X(x).
\end{equation}
Combining \eqref{hhom} and \eqref{Fi}, we have
\begin{equation}\label{Fdef}
X(x)=\sum_{i} h_{i}(\Omega,x)e_{i}+h(\Omega,x)x=\nabla h(\Omega,x)+h(\Omega,x)x.
\end{equation}
Here $\nabla$ is the spherical gradient. On the other hand, since we can extend $h(\Omega,x)$ to $\rnnn$ as a 1-homogeneous function $h(\Omega, \cdot)$, then restrict the gradient of $h(\Omega,\cdot)$ on $\sn$, it yields that
\begin{equation}\label{hf}
\overline{\nabla} h(\Omega,x)=X(x), \ \forall x\in{\sn},
\end{equation}
where $\overline{\nabla}$ is the gradient operator in $\rnnn$. Let $h_{ij}$ be the second order covariant derivatives of $h(\Omega,\cdot)$ on ${\sn}$ with respect to a local orthonormal frame. Then, applying \eqref{Fdef} and \eqref{hf}, we have the following equalities:
\begin{equation}\label{hgra}
\overline{\nabla} h(\Omega,x)=\sum_{i}h_{i}e_{i}+hx, \quad X_{i}(x)=\sum_{j}(h_{ij}+h\delta_{ij})e_{j}.
\end{equation}
The Gauss curvature of $\partial \Omega$ at $X(x)$ is given by
\begin{equation*}
\kappa=\frac{1}{det(\nabla^{2}h+hI)},
\end{equation*}
where $\nabla^{2}h$ denotes the spherical Hessian matrix of $h$, and $I$ is the identity matrix.

The radial function $\rho$ of $\Omega$ is given by
\[
\rho(u)=\max\{\lambda> 0: \lambda u\in \Omega\},\quad \forall u\in {\sn}.
\]
It is clear to see that $\rho(u)u\in \partial \Omega$ for any $u\in \sn$, then the Gauss map $g$ of $\partial \Omega$ can be expressed as
\begin{equation*}\label{}
g(\rho(u)u)=\frac{\rho(u)u-\nabla \rho}{\sqrt{\rho^{2}+|\nabla \rho|^{2}}}.
\end{equation*}
Let $u$ and $x$ be related by the following equality:
\begin{equation}\label{hprl}
\rho(u)u=\overline{\nabla }h(\Omega,x)=X(x)=\nabla h+hx,
\end{equation}
then we obtain (see, e.g., \cite{GLM09})
\begin{equation}\label{xur}
x=\frac{\rho(u)u-\nabla \rho}{\sqrt{\rho^{2}+|\nabla \rho|^{2}}},\quad u=\frac{\nabla h+hx}{\sqrt{|\nabla h|^{2}+h^{2}}},
\end{equation}
and the following formula is clear (see for example \cite{LSW200})
\begin{equation}\label{hp}
\frac{h(x)}{\kappa(x)}dx=\rho^{n}(u)du.
\end{equation}

\section{The geometric flow and relevant functional}
\label{Sec3}
In this section, we are in the place to introduce the geometric flow and the relevant functional.

Let $\Omega_{0}$ be a smooth, origin symmetric and strictly convex body in $\rnnn$, $f: \sn\rightarrow (0,\infty)$ be a smooth function , as presented above, we are concerned with a family of convex hypersurfaces $\partial \Omega_{t}$ parameterized by smooth map $X(\cdot,t):\sn\rightarrow \rnnn$ satisfying the following flow equation,
\begin{equation}\label{xOrflow}
\left\{
\begin{array}{lr}
\frac{\partial X(x,t)}{\partial t}=-\frac{f(\nu)|X|^{n}\kappa \nu \eta(t)}{\varphi(1/(X\cdot \nu)) }+X(x,t);  \\
X(x,0)=X_{0}(x),
\end{array}\right.
\end{equation}
where $\kappa$ is the Gauss curvature of the hypersurface $\partial \Omega_{t}$ at $X(\cdot,t)$, $\nu=x$ is the unit outer normal vector of $\partial\Omega_{t}$ at $X(\cdot,t)$, and for any $u\in \sn$, $\eta(t)$ is defined by
\begin{equation}\label{yi}
\eta(t)=\frac{\int_{\sn}du}{\int_{\sn}\frac{f(x)}{\varphi(1/(X\cdot \nu))}dx}.
\end{equation}
Taking the scalar product of both sides of the equation and the of the initial condition in \eqref{xOrflow} by $\nu$, by means of the definition of support function, we describe the flow equation associated with the support function as
\begin{equation}\label{hOrflow}
\left\{
\begin{array}{lr}
\frac{\partial h(x,t)}{\partial t}=-\frac{f(x)\rho(u,t)^{n}\kappa  \eta(t)}{\varphi(1/h) }+h(x,t);  \\
h(x,0)=h_{0}(x),
\end{array}\right.
\end{equation}
where $\rho(u,t)$ is the radial function of $\Omega_{t}$ for any $u\in \sn$ satisfying \eqref{hprl}, this together with the definition of support function \eqref{hhom}, then the following relation between $x$ and $u$ can be deduced:
\begin{equation}\label{ph}
\rho(u,t)(u\cdot x)=h(x,t).
\end{equation}
 Applying \eqref{ph}, as shown in \cite{CHZ19}, we obtain
\begin{equation}\label{pht}
\partial_{t}\rho(u,t)=\frac{\rho(u,t)}{h(x,t)}\partial_{t}h(x,t).
\end{equation}
Combining \eqref{hOrflow} with  \eqref{pht}, it is not hard to see that $\rho(u,t)$ satisfies the following flow equation,
\begin{equation}\label{pOrflow}
\left\{
\begin{array}{lr}
\frac{\partial \rho(u,t)}{\partial t}=-\frac{f(x)\rho^{n+1}\kappa  \eta(t)}{\varphi(1/h) h(x,t) }+\rho(u,t);  \\
\rho(u,0)=\rho_{0}(u).
\end{array}\right.
\end{equation}
In the study of the flow \eqref{xOrflow}, it is essential to derive the following results.
\begin{lem}\label{logcons}
Under the assumptions of Theorem \ref{main*}. Let $\Omega_{t}$ be a smooth, origin symmetric, strictly convex solution to the flow \eqref{xOrflow}. Then
\begin{equation*}
\int_{\sn}{\rm log}\rho(u,t)du=constant
\end{equation*}
for $t\geq 0$. Here $\rho(\cdot,t)$ is the radial function of $\Omega_{t}$.
\end{lem}
\begin{proof}
Using  \eqref{hp}, \eqref{yi} and \eqref{pOrflow}, we have
\begin{equation}
\begin{split}
\label{}
\frac{d}{dt}\int_{\sn}{\rm log}\rho(u,t)du&=\int_{\sn}\frac{1}{\rho(u,t)}\frac{d \rho(u,t)}{dt}du\\
&=\int_{\sn}\frac{1}{\rho(u,t)}\left(\rho(u,t)-\frac{f(x)\rho(u,t)^{n+1}\kappa }{\varphi(1/h)h(x,t)}\frac{\int_{\sn}du}{\int_{\sn}\frac{f(x)}{\varphi(1/h)}dx}\right)du\\
&=\int_{\sn}du-\int_{\sn}\frac{f(x)}{\varphi(1/h)h(x,t)}\rho(u,t)^{n}\kappa du\frac{\int_{\sn}du}{\int_{\sn}\frac{f(x)}{\varphi(1/h)}dx}\\
&=\int_{\sn}du-\int_{\sn}\frac{f(x)}{\varphi(1/h)h(x,t)}h(x,t)dx\frac{\int_{\sn}du}{\int_{\sn}\frac{f(x)}{\varphi(1/h)}dx}\\
&=\int_{\sn}du-\int_{\sn}du\\
&=0.
\end{split}
\end{equation}
This completes the proof.
\end{proof}
Now, we set the relevant functional associated with the flow \eqref{xOrflow} as:
\begin{equation*}
P(t)=\int_{\sn}\Re(1/h(x,t))f(x)dx, \quad for \ \quad t\geq 0,
\end{equation*}
which will be equipped with the monotonicity along the flow \eqref{xOrflow}. It is shown in the following Lemmas.
\begin{lem}\label{non-decre}
If the condition $(i)$ of Theorem \ref{main*} holds, then along the flow \eqref{xOrflow}, $P(t)$ is non-decreasing.
\end{lem}
\begin{proof}
Due to $\Re'(1/h)=\frac{h}{\varphi(1/h)}$, by virtue of \eqref{yi} and \eqref{hOrflow}, we have
\begin{equation*}
\begin{split}
\label{}
P^{'}(t)&=\int_{\sn}-\frac{f}{h(x,t)\varphi(1/h)}\frac{\partial h(x,t)}{\partial t}dx\\
&=\int_{\sn}-\frac{f}{h(x,t)\varphi(1/h)}\left(h(x,t)-\frac{f\rho^{n}\kappa \eta(t)}{\varphi(1/h)}\right)dx\\
&=\int_{\sn}-\frac{f}{\varphi(1/h)}dx+\int_{\sn}\frac{f^{2}\rho^{n}\kappa}{h \varphi(1/h)^{2}}dx \frac{\int_{\sn}du}{\int_{\sn}\frac{f}{\varphi(1/h)}dx},
\end{split}
\end{equation*}
which turns into
\begin{equation}
\begin{split}
\label{derpt}
P^{'}(t)\int_{\sn}\frac{f}{\varphi(1/h)}dx&=-\left(\int_{\sn}\frac{f}{\varphi(1/h)}dx\right)^{2}+\int_{\sn}\frac{f^{2}\rho^{n}\kappa}{h \varphi(1/h)^{2}}dx\int_{\sn}du\\
&=-\left(\int_{\sn}\frac{f}{\varphi(1/h)}dx\right)^{2}+\int_{\sn}\frac{f^{2}\rho^{n}\kappa}{h \varphi(1/h)^{2}}dx\int_{\sn}\frac{h}{\kappa \rho^{n}}dx\\
&=-\left(\int_{\sn}\sqrt{\frac{f^{2}\rho^{n}\kappa}{h\varphi(1/h)^{2}}}\sqrt{\frac{h}{\kappa \rho^{n}}}dx\right)^{2}+\int_{\sn}\frac{f^{2}\rho^{n}\kappa}{h \varphi(1/h)^{2}}dx\int_{\sn}\frac{h}{\kappa \rho^{n}}dx \\
&\geq 0,
\end{split}
\end{equation}
where the last inequality is due to the H\"{o}lder inequality. Therefore, $P(t)$ is non-decreasing.
\end{proof}

\begin{lem}\label{non-incre}
If the condition $(ii)$ of Theorem \ref{main*} holds, then along the flow \eqref{xOrflow}, $P(t)$ is non-increasing.
\end{lem}
\begin{proof}
The proof is similar to Lemma \ref{non-decre}. Since $\Re'(1/h)=-\frac{h}{\varphi(1/h)}$, we have
\begin{equation}
\begin{split}
\label{incrpt}
P^{'}(t)\int_{\sn}\frac{f}{\varphi(1/h)}dx&=\left(\int_{\sn}\frac{f}{\varphi(1/h)}dx\right)^{2}-\int_{\sn}\frac{f^{2}\rho^{n}\kappa}{h \varphi^{2}}dx\int_{\sn}du\\
&=\left(\int_{\sn}\frac{f}{\varphi(1/h)}dx\right)^{2}-\int_{\sn}\frac{f^{2}\rho^{n}\kappa}{h \varphi(1/h)^{2}}dx\int_{\sn}\frac{h}{\kappa \rho^{n}}dx\\
&=\left(\int_{\sn}\sqrt{\frac{f^{2}\rho^{n}\kappa}{h\varphi(1/h)^{2}}}\sqrt{\frac{h}{\kappa \rho^{n}}}dx\right)^{2}-\int_{\sn}\frac{f^{2}\rho^{n}\kappa}{h \varphi(1/h)^{2}}dx\int_{\sn}\frac{h}{\kappa \rho^{n}}dx\\
&\leq 0.
\end{split}
\end{equation}
 Therefore, $P(t)$ is non-increasing.
\end{proof}

\section{$C^{0},C^{1}$ estimates}
\label{Sec4}
In this section, we shall obtain the $C^{0}$, $C^{1}$ estimates of solutions to the flow \eqref{xOrflow}. Let us begin with completing the $C^{0}$ estimate.

\begin{lem}\label{C0}
Under the assumptions of Theorem \ref{main*}. Let $\Omega_{t}$ be a smooth, origin symmetric, and strictly convex solution to the flow \eqref{xOrflow}. Then, there exists a positive constant $C$, independent of $t$, such that
\begin{equation}\label{C0*}
\frac{1}{C}\leq h(x,t)\leq C, \ \forall (x,t)\in {\sn}\times (0,\infty),
\end{equation}
and
\begin{equation}\label{C00*}
\frac{1}{C}\leq \rho(u,t)\leq C, \ \forall (u,t)\in {\sn}\times (0,\infty).
\end{equation}
Here $h(x,t)$ and $\rho(u,t)$ are the support function and the radial function of $\Omega_{t}$. Furthermore, $\eta(t)$ is uniformly bounded above, and below away from zero.
\end{lem}
\begin{proof}
In view of \cite[Lemma 2.6]{CL21}, we know
\begin{equation}\label{hpre}
\min_{\sn}h(x,t)=\min_{\sn}\rho(u,t), \quad \max_{\sn}h(x,t)= \max_{\sn}\rho(u,t).
\end{equation}
\eqref{hpre} illustrates that \eqref{C0*} and \eqref{C00*} are equivalent. So, for upper bound (or lower bound), we only need to establish \eqref{C0*} or \eqref{C00*}.

Our aim is to prove \eqref{C0*}. We first prove the upper bound of \eqref{C0*},  on the one hand, if the condition $(i)$ of Theorem \ref{main*} holds, assume that $R(t):=\max_{\sn} h(x,t)$ is attained at the north pole $\tilde{x}=(0,\cdots,0,1)$ at $t> 0$. By virtue of the convexity of $\Omega_{t}$ and the definition of $h(x,t)$, it follows that (see e.g., \cite[Lemma 2.6]{CL21})
\begin{equation}\label{convex}
h(x,t)\geq \max_{\sn}\rho(u,t)|u\cdot x| =\max_{\sn}h(x,t)|\tilde{x}\cdot x| =R(t)|x_{n}|,
\end{equation}
where $x_{n}$ is $n$-th coordinate component of $x$. Let $\varepsilon$ be a small  positive constant, then we set
 \[
\mathcal{O}^{*}=\{x\in \sn:|x_{n}|< \varepsilon\}.
 \]
Suppose $R(t)>1$, as $P$ is non-decreasing and $\Re\geq 0$ is strictly increasing, from \eqref{convex}, we obtain
\begin{equation}
\begin{split}
\label{PTT}
P(0)&\leq P(t)\\
&=\int_{\sn\backslash\mathcal{O}^{*}}\Re(1/h(x,t))f(x)dx+\int_{ \mathcal{O}^{*}}\Re(1/h(x,t))f(x)dx\\
&\leq (\max_{\sn} f)\Re (1/(\varepsilon R(t)))|\sn \backslash \mathcal{O}^{*}|+\int_{\mathcal{O}^{*}}\Re (1/(R(t)|x_{n}|))f(x)dx\\
& \leq (\max_{\sn} f)\Re (1/(\varepsilon R(t)))|\sn \backslash \mathcal{O}^{*}|+\int_{\sn}f(x)\Re(1/|x_{n}|)dx.
\end{split}
\end{equation}
Note that $\lim_{s\rightarrow 0}\Re(s)=0$, in view of \eqref{PTT}, as $R\rightarrow \infty$, we obtain
\begin{equation}\label{P0C}
P(0)\leq \hat{C}.
\end{equation}
 By our choice of $\Omega_{0}$, satisfying $P(0)>\hat{C}$, then \eqref{P0C} is violated. So the upper bound of $h(x,t)$ is obtained.

  On the other hand, if the condition $(ii)$ of Theorem \ref{main*} holds.  Supposing again $R(t):=\max_{\sn} h(x,t)$ is attained at the north pole $\tilde{x}=(0,\cdots,0,1)$ at $t> 0$. With the aid of again the convexity of $\Omega_{t}$, we have
  \[
  h(x,t)\geq R(t)|x_{n}|.
  \]
  Since $\Re\geq 0$ is strictly decreasing and $P(t)$ is non-increasing in time $t$, we obtain
 \begin{equation}
\begin{split}
\label{JJI}
P(0)&\geq P(t)\\
&\geq \int_{\sn}\Re(1/(R(t)|x_{n}|))f(x)dx\\
&\geq (\min_{\sn}f)\int_{\sn}\Re(1/(R(t)|x_{n}|))dx.
\end{split}
\end{equation}
  Denote $\mathcal{\tilde{O}}=\{x\in \sn: |x_{n}|\geq\frac{1}{2}\}$, then \eqref{JJI} becomes
\begin{equation*}
\begin{split}
\label{Up3}
P(0)&\geq (\min_{\sn}f)\int_{\{|x_{n}|\geq\frac{1}{2}\}\cap {\sn}}\Re(2/R(t))dx\\
&=(\min_{\sn}f)\Re(2/R(t))|\mathcal{\tilde{O}}|.
\end{split}
\end{equation*}
By virtue of $\lim_{s\rightarrow 0}\Re(s)=\infty$, it illustrates that $h(x,t)$ is uniformly bounded above. So we obtain the upper bound of \eqref{C0*}.

For the uniform lower bound of $h$, we argue by contradiction. Let $\{t_{k}\}\subset [0,\infty)$ be a sequence such that $h(x,t_{k})$ is not uniformly bounded away from 0, i.e.,
\[
\min_{\sn}h(\cdot, t_{k})\rightarrow 0 \quad as \quad k\rightarrow\infty.
\]
With the aid of the upper bound of \eqref{C0*},  using Blaschke selection theorem (see \cite{S14}), then there is a sequence in $\{\Omega_{t_{k}}\}$, which is still denoted by $\{\Omega_{t_{k}}\}$, such that
\[
\Omega_{t_{k}}\rightarrow \widetilde{\Omega}\quad as \quad k\rightarrow\infty.
\]
Since $\Omega_{t_{k}}$ is an origin-symmetric convex body, $\widetilde{\Omega}$ is also origin-symmetric. Then, we have
\[
\min_{\sn}h_{\widetilde{\Omega}}=\lim_{k\rightarrow\infty}\min_{\sn} h_{\Omega_{t_{k}}}=0.
\]
This implies that $\widetilde{\Omega}$ is contained in a hyperplane in $\rnnn$. Then, we have
\begin{equation}\label{op}
\rho_{\widetilde{\Omega}}=0\quad a.e.\  {\rm in} \ \sn.
\end{equation}
By virtue of \eqref{op} and Lemma \ref{logcons},  for any $\varepsilon>0$, we have
\begin{equation*}
\begin{split}
\label{Up3}
 \int_{\sn}\log \rho_{\Omega_{0}}du&=\int_{\sn}\log \rho_{\Omega_{t_{k}}}du\\
&\leq \lim_{k\rightarrow \infty}\int_{\sn}\log[\rho_{\Omega_{t_{k}}}+\varepsilon]du\\
&=\int_{\sn}\log \varepsilon du\\
&\rightarrow -\infty\quad as \quad \varepsilon\rightarrow 0,
\end{split}
\end{equation*}
which is a contradiction. Then, we have
\[
\min_{\sn \times (0,\infty)} h(x,t)\geq C
\]
for some positive constant $C$, independent of $t$. The lower bound of \eqref{C0*} follows.  The lower and upper bounds on $h$ imply bounds on $\eta(t)$. Hence, we complete the proof.
\end{proof}

The $C^{1}$ estimate is as follows.
\begin{lem}\label{C1} Under the assumptions of Theorem \ref{main*}. Let $\Omega_{t}$ be a smooth, origin symmetric, and strictly convex solution to the flow \eqref{xOrflow}. Then, there exists a positive constant $C$, independent of $t$, such that
\begin{equation*}\label{C1*}
|\nabla h(x,t)|\leq C, \ \forall (x,t)\in {\sn}\times (0,\infty),
\end{equation*}
and
\begin{equation*}\label{C11*}
|\nabla \rho(u,t)|\leq C, \ \forall (u,t)\in {\sn}\times (0,\infty)
\end{equation*}
for some $C> 0$, independent of $t$ .
\end{lem}
\begin{proof}
By means of \eqref{hhom}, \eqref{hprl} and \eqref{xur}, we obtain the following equalities:
\begin{equation}\label{phph}
\rho^{2}=h^{2}+|\nabla h|^{2},  \quad h=\frac{\rho^{2}}{\sqrt{\rho^{2}+|\nabla \rho|^{2}}}.
\end{equation}
Then \eqref{phph} and Lemma \ref{C0} lead to this lemma directly.
\end{proof}

\section{$C^{2}$ estimate}
\label{Sec5}
Utilizing the above $C^{0}, C^{1}$ estimates, the upper and lower bounds of principal curvatures will be derived. It is shown in the following result.
\begin{theo}\label{PCUL*}
 Under the assumptions of Theorem \ref{main*}. Let $\Omega_{t}$ be a smooth, origin symmetric and strictly convex solution to the flow \eqref{xOrflow}. Then there exists a positive constant $C$, independent of $t$, such that the principal curvatures $\kappa_{i}$ of $\Omega_{t}$, $i=1,\ldots,n-1$, are bounded from above and below, satisfying
\begin{equation}\label{PUL}
\frac{1}{C}\leq \kappa_{i}\leq C, \ \forall(x,t)\in {\sn}\times (0,\infty).
\end{equation}
\end{theo}

\begin{proof}
First, we shall prove the upper bound of Gauss curvature $\kappa$. It is essential to construct the following auxiliary function,
\begin{equation}\label{AF1}
Q(x,t)=\frac{\frac{f(x)\rho^{n}\kappa}{\varphi(1/h)} \frac{\int_{\sn}du}{\int_{\sn}\frac{f}{\varphi(1/h)}{dx}}-h(x,t)}{h-\varepsilon_{0}}=\frac{-h_{t}}{h-\varepsilon_{0}},
\end{equation}
where
\begin{equation*}\label{AF2}
\varepsilon_{0}=\frac{1}{2}\min_{{\sn}\times (0,\infty)}h(x,t)>0.
\end{equation*}
For any fixed $t\in (0,\infty)$, assume that the maximum of $Q(x,t)$ is achieved at $x_{0}$. Rotate the axes so that $\{w_{ij}\}$ is diagonal at $x_{0}$, where $w_{ij}:=h_{ij}+h\delta_{ij}$. Thus, we obtain that at $x_{0}$,
\begin{equation}\label{Up1}
0=\nabla_{i}Q=\frac{-h_{ti}}{h-\varepsilon_{0}}+\frac{h_{t}h_{i}}{(h-\varepsilon_{0})^{2}}.
\end{equation}
Then, by virtue of \eqref{Up1}, at $x_{0}$, we also obtain
\begin{equation}
\begin{split}
\label{Up2}
0\geq\nabla_{ii}Q&=\frac{-h_{tii}}{h-\varepsilon_{0}}+\frac{2h_{ti}h_{i}+h_{t}h_{ii}}{(h-\varepsilon_{0})^{2}}-\frac{2h_{t}h^{2}_{i}}{(h-\varepsilon_{0})^{3}}\\
&=\frac{-h_{tii}}{h-\varepsilon_{0}}+\frac{h_{t}h_{ii}}{(h-\varepsilon_{0})^{2}}.
\end{split}
\end{equation}
 Then, \eqref{Up2} tells
\begin{equation}
\begin{split}
\label{Up3}
-h_{tii}-h_{t}&\leq - \frac{h_{t}h_{ii}}{h-\varepsilon_{0}}-h_{t}\\
&=\frac{-h_{t}}{h-\varepsilon_{0}}[h_{ii}+(h-\varepsilon_{0}))]\\
&=Q(w_{ii}-\varepsilon_{0}).
\end{split}
\end{equation}
 Furthermore, applying \eqref{hOrflow} and \eqref{AF1}, we have
\begin{equation}
\begin{split}
\label{Up4*}
\partial_{t}Q&=\frac{-h_{tt}}{h-\varepsilon_{0}}+\frac{h^{2}_{t}}{(h-\varepsilon_{0})^{2}}\\
&=\frac{f}{h-\varepsilon_{0}}\left[\frac{\partial [det(\nabla ^{2}h+hI)]^{-1}}{\partial t}\frac{\rho^{n}}{\varphi(1/h)}\frac{\int_{\sn}du}{\int_{\sn}\frac{f}{\varphi(1/h)}dx}+\frac{\rho^{n}}{\varphi(1/h)}\kappa \frac{\partial \left[ \frac{\int_{\sn}du}{\int_{\sn}\frac{f}{\varphi(1/h)}dx} \right]}{\partial t} \right.\\
&\quad \left. +\kappa \frac{\partial_{t}\left(\frac{\rho^{n}}{\varphi(1/h)}\right)}{\partial t}\frac{\int_{\sn}du}{\int_{\sn}\frac{f}{\varphi(1/h)}dx} \right]
+Q+ Q^{2}.\\
\end{split}
\end{equation}
In light of \eqref{Up4*}, by using \eqref{Up3}, at $x_{0}$, one has
\begin{equation*}
\begin{split}
\label{}
\frac{\partial [det(\nabla^{2}_{\sn}h+hI)]^{-1}}{\partial t}&=-[det(\nabla^{2}_{\sn}h+hI)]^{-2}\sum_{i}\frac{\partial[det(\nabla^{2}_{\sn}h+hI)] }{\partial w_{ii}}(h_{tii}+h_{t})\\
&\leq[det(\nabla^{2}_{\sn}h+hI)]^{-2}\sum_{i}\frac{\partial[det(\nabla^{2}_{\sn}h+hI)] }{\partial w_{ii}}Q(w_{ii}-\varepsilon_{0})\\
&=\kappa Q[(n-1)-\varepsilon_{0}\sum_{i} w^{ii}],
\end{split}
\end{equation*}
where $\{w^{ij}\}$ is the inverse matrix of $\{w_{ij}\}$. Recall the fact that the eigenvalue of $\{w_{ij}\}$ and $\{w^{ij}\}$ are respectively the principal radii and principal curvature of $\partial \Omega_{t}$ (see for example \cite{U91}). Then we have
\begin{equation}
\begin{split}
\label{Up5}
\frac{\partial [det(\nabla^{2}_{\sn}h+hI)]^{-1}}{\partial t}&=\kappa Q[(n-1)-\varepsilon_{0}H]\\
&\leq \kappa Q[(n-1)-\varepsilon_{0}(n-1)\kappa^{\frac{1}{n-1}}],
\end{split}
\end{equation}
where $H$ denotes the mean curvature of $\partial\Omega_{t}$, and the last inequality stems from  $H\geq (n-1)(\Pi_{i} w^{ii})^\frac{1}{n-1}=(n-1)\kappa^{\frac{1}{n-1}}$.

In addition,
\begin{equation}
\begin{split}
\label{Up6}
\frac{\partial \left[ \frac{\int_{\sn}du}{\int_{\sn}\frac{f}{\varphi(1/h)}dx} \right]}{\partial t}&=-\frac{\int_{\sn}du}{\left(\int_{\sn}\frac{f}{\varphi(1/h)}dx\right)^{2}}\int_{\sn}\frac{f}{h^{2}\varphi(1/h)^{2}}\varphi^{'}(1/h)\frac{\partial h}{\partial t}dx\\
&=\frac{\int_{\sn}du}{\left(\int_{\sn}\frac{f}{\varphi(1/h)}dx\right)^{2}}\int_{\sn}\frac{f}{h^{2}\varphi(1/h)^{2}}\varphi^{'}(1/h)(h-\varepsilon_{0})Qdx\\
&\leq (\max_{h\in I_{(0,\infty)}}|\varphi^{'}(1/h)|)Q(x_{0},t)\frac{\int_{\sn}du}{\left(\int_{\sn}\frac{f}{\varphi(1/h)}dx\right)^{2}}\int_{\sn}\frac{f (h-\varepsilon_{0})}{h^{2}\varphi(1/h)^{2}}dx,
\end{split}
\end{equation}
where $I_{(0,\infty)}$ is a bounded interval depending only on the upper and lower bounds of $h$ on $(0,\infty)$, and $\varphi^{'}(1/h)$ denotes $\frac{\partial \varphi(1/h)}{\partial (1/h)}$, then using the first equality of \eqref{phph} and \eqref{Up1}, we have
\begin{equation}
\begin{split}
\label{Up7}
\frac{\partial\left( \frac{\rho^{n}}{\varphi(1/h)}\right)}{\partial t}&=\frac{n\rho^{n-2}}{\varphi(1/h)}\left(hh_{t}+\sum_{k} h_{k}h_{kt}\right)+\frac{\rho^{n}}{h^{2}\varphi(1/h)^{2}}\varphi^{'}(1/h)\frac{\partial h}{\partial t}\\
&=\frac{n\rho^{n-2}}{\varphi(1/h)}Q(\varepsilon_{0}h-\rho^{2})-\frac{\rho^{n}}{h^{2}\varphi(1/h)^{2}}\varphi^{'}(1/h)(h-\varepsilon_{0})Q\\
&\leq \frac{n\rho^{n-2}}{\varphi(1/h)}Q(\varepsilon_{0}h-\rho^{2})+( \max_{h\in I_{(0,\infty)}} |\varphi^{'}(1/h)|)\frac{\rho^{n}}{h^{2}\varphi(1/h)^{2}}(h-\varepsilon_{0})Q.
\end{split}
\end{equation}
Substituting \eqref{Up5}, \eqref{Up6} and \eqref{Up7} into \eqref{Up4*},  we obtain that at $x_{0}$,
\begin{equation}
\begin{split}
\label{finic2}
\partial_{t}Q
&\leq \frac{f}{h-\varepsilon_{0}}\left[\kappa Q[(n-1)-\varepsilon_{0}(n-1)\kappa^{\frac{1}{n-1}}]\frac{\rho^{n}}{\varphi(1/h)}\frac{\int_{\sn}du}{\int_{\sn}\frac{f}{\varphi(1/h)}dx}\right.\\
&\quad \left.+\frac{\rho^{n}}{\varphi(1/h)}\kappa (\max_{h\in I_{(0,\infty)}}|\varphi^{'}(1/h)|)Q\frac{\int_{\sn}du}{\left(\int_{\sn}\frac{f}{\varphi(1/h)}dx\right)^{2}}\int_{\sn}\frac{f (h-\varepsilon_{0})}{h^{2}\varphi(1/h)^{2}}dx \right.\\
 &\left.\quad +\kappa \left(\frac{n\rho^{n-2}}{\varphi(1/h)}Q(\varepsilon_{0}h-\rho^{2})+(\max_{h\in I_{(0,\infty)}} |\varphi^{'}(1/h)|)\frac{\rho^{n}}{h^{2}\varphi(1/h)^{2}}(h-\varepsilon_{0})Q\right)\frac{\int_{\sn}du}{\int_{\sn}\frac{f}{\varphi(1/h)}dx} \right]
+Q+ Q^{2}.\\
\end{split}
\end{equation}
The a priori estimates in Section \ref{Sec4} and equation \eqref{AF1} allow us to assume $\kappa\approx Q>>1$. Then, using \eqref{finic2}, we conclude
\begin{equation}\label{Up13}
\partial_{t}Q\leq C_{0}Q^{2}(C_{1}-\varepsilon_{0}Q^{\frac{1}{n-1}})<0
\end{equation}
for some $C_{0}$, $C_{1}$ only depending on $\min_{\sn} f$, $\max_{\sn} f$, $||\varphi||_{C^{1}(I_{(0,\infty)})}$, $||h||_{C^{0}(\sn\times (0,\infty))}$, $||\rho||_{C^{0}(\sn\times (0,\infty))}$, $\min_{(\sn\times (0,\infty))} h$ and $\min_{I_{(0,\infty)}} \varphi$. Hence, the ODE \eqref{Up13} tells that
\begin{equation}\label{Qbod}
Q(x_{0},t)\leq C
\end{equation}
for some $C>0$, independent of $t$.

Making use of the priori estimates in Section \ref{Sec4} and Section \ref{Sec5}, \eqref{AF1} and \eqref{Qbod}, for any $(x,t)$, we obtain
\begin{equation}\label{UpFal2}
\kappa=\frac{(h-\varepsilon_{0})Q(x,t)+h}{f(x)\frac{\rho^{n}}{\varphi(1/h)}\frac{\int_{\sn}du}{\int_{\sn}\frac{f}{\varphi(1/h)}dx}}\leq \frac{(h-\varepsilon_{0})Q(x_{0},t)+h}{f(x)\frac{\rho^{n}}{\varphi(1/h)}\frac{\int_{\sn}du}{\int_{\sn}\frac{f}{\varphi(1/h)}dx}}\leq C
\end{equation}
for some $C> 0$, independent of $t$. So, the upper bound of Gauss curvature is established.

We are now in a position to prove the lower bound of \eqref{PUL}. We introduce the auxiliary function
\begin{equation}\label{LFun}
E(x,t)={\rm log}\lambda_{max}(\{w_{ij}\})-d{\rm log} h(x,t)+l|\nabla  h|^{2},
\end{equation}
where $d$ and $l$ are positive constants to be specified later, $\lambda_{max}(\{w_{ij}\})$ is the maximal eigenvalue of $\{w_{ij}\}$. As showed in above, one can know that the eigenvalue of $\{w_{ij}\}$ and $\{w^{ij}\}$ are respectively the principal radii and principal curvature of $\partial \Omega_{t}$.

For any fixed $t\in(0,\infty)$, suppose that the maximum of $E(x,t)$ is attained at $x_{0}$ on ${\sn}$. By a rotation of coordinates, we may assume that $\{w_{ij}(x_{0},t)\}$ is diagonal, and $\lambda_{max}(\{w_{ij}(x_{0},t)\})=w_{11}(x_{0},t)$. To obtain the lower bound of principal curvature, it is necessary to get the upper bound of $w_{11}$.  By means of the above assumption, we transform \eqref{LFun} into
\begin{equation}\label{LFun2}
\widetilde{E}(x,t)={\rm log}w_{11}-d{\rm log} h(x,t)+l|\nabla h|^{2}.
\end{equation}
Utilizing again the above assumption, thus, for any fixed $t\in(0,\infty)$, $\widetilde{E}(x,t)$ has a local maximum at $x_{0}$, which implies that, at $x_{0}$, it yields
\begin{equation}
\begin{split}
\label{gaslw1}
0=\nabla_{i}\widetilde{E}&=w^{11}\nabla_{i}w_{11}-d\frac{h_{i}}{h}+2l \sum_{j} h_{j}h_{ji}\\
&=w^{11}(w_{1i1})-d\frac{h_{i}}{h}+2lh_{i}h_{ii},
\end{split}
\end{equation}
and
\begin{equation}\label{gaslw2}
0\geq \nabla_{ii}\widetilde{E}=w^{11}\nabla_{ii}w_{11}-(w^{11})^{2}(\nabla_{i}w_{11})^{2}-d\left(\frac{h_{ii}}{h}-\frac{h^{2}_{i}}{h^{2}}\right)
+2l\left[\sum_{j} h_{j}h_{jii}+h^{2}_{ii}\right].
\end{equation}
Furthermore, we have
\begin{equation}
\label{gaslw3}
\partial_{t}\widetilde{E}=w^{11}\partial_{t}w_{11}-d\frac{h_{t}}{h}+2l \sum_{j} h_{j}h_{jt}=w^{11}(h_{11t}+h_{t})-d\frac{h_{t}}{h}+2l \sum_{j} h_{j}h_{jt}.
\end{equation}
On the other hand, applying \eqref{yi} and \eqref{hOrflow}, we have
\begin{equation}
\begin{split}
\label{gaslw4}
{\rm log}(h-h_{t})&={\rm log}\left(f\frac{\rho^{n}}{\varphi(1/h)}\kappa \frac{\int_{\sn}du}{\int_{\sn}\frac{f}{\varphi(1/h)}dx}\right)\\
&=-{\rm log}det(\nabla ^{2}h+hI)+\chi(x,t),
\end{split}
\end{equation}
where
\begin{equation}\label{gaslw5}
\chi(x,t):={\rm log}\left(f\frac{\rho^{n}}{\varphi(1/h)} \frac{\int_{\sn}du}{\int_{\sn}\frac{f}{\varphi(1/h)}dx}\right).
\end{equation}
Now, taking the covariant derivative of \eqref{gaslw4} with respect to $e_{j}$, it follows that
\begin{align}\label{gaslw6}
\frac{h_{j}-h_{jt}}{h-h_{t}}&=-\sum_{i,k} w^{ik}\nabla_{j}w_{ik}+\nabla_{j}\chi\notag\\
&=-\sum_{i} w^{ii}(h_{jii}+h_{i}\delta_{ij})+\nabla_{j}\chi,
\end{align}
and
\begin{equation}
\begin{split}
\label{gaslw7}
&\frac{h_{11}-h_{11t}}{h-h_{t}}-\frac{(h_{1}-h_{1t})^{2}}{(h-h_{t})^{2}}\\
&=-\sum_{i}w^{ii}\nabla_{11}w_{ii}+\sum_{i,k} w^{ii}w^{kk}(\nabla_{1}w_{ik})^{2}+\nabla_{11}\chi.
\end{split}
\end{equation}
On the other hand, the Ricci identity on sphere reads
\begin{equation*}
\nabla_{11}w_{ij}=\nabla_{ij}w_{11}-\delta_{ij}w_{11}+\delta_{11}w_{ij}-\delta_{1i}w_{1j}+\delta_{1j}w_{1i}.
\end{equation*}
This together with  \eqref{gaslw2}, \eqref{gaslw3}, \eqref{gaslw6} and \eqref{gaslw7}, by a direct computation (see also \cite{CHZ19}), we have at $x_{0}$,
\begin{equation}
\begin{split}
\label{gaslw8}
\frac{\partial_{t}\widetilde{E}}{h-h_{t}}
&=w^{11}\left[\frac{(h_{11t}-h_{11}+h_{11}+h-h+h_{t})}{h-h_{t}}\right]-d\frac{1}{h}\frac{h_{t}-h+h}{(h-h_{t})}+2l\frac{\sum_{j} h_{j}h_{jt}}{h-h_{t}}\\
&=w^{11}\left[-\frac{(h_{1}-h_{1t})^{2}}{(h-h_{t})^{2}}+\sum_{i} w^{ii}\nabla_{11}w_{ii}-\sum_{i,k} w^{ii}w^{kk}(\nabla_{1}w_{ik})^{2}-\nabla_{11}\chi\right]\\
&\quad +\frac{1}{h-h_{t}}-w^{11}+\frac{d}{h}-\frac{d}{h-h_{t}}+2l\frac{\sum_{j} h_{j}h_{jt}}{h-h_{t}}\\
&\leq w^{11}\left[\sum_{i} w^{ii}(\nabla_{ii}w_{11}-w_{11}+w_{ii})-\sum_{i,k} w^{ii}w^{kk}(\nabla_{1}w_{ik})^{2}\right]\\
&\quad+\frac{1}{h-h_{t}}(1-d)-w^{11}\nabla_{11}\chi+\frac{d}{h}+2l\frac{\sum_{j} h_{j}h_{jt}}{h-h_{t}}\\
&\leq \sum_{i} w^{ii}\left[ (w^{11})^{2}(\nabla_{i}w_{11})^{2}+d\left(\frac{h_{ii}}{h}-\frac{h^{2}_{i}}{h^{2}}\right)-2l\left(\sum_{j} h_{j}h_{jii}+h^{2}_{ii}\right)\right]\\
&\quad-w^{11}\sum_{i,k} w^{ii}w^{kk}(\nabla_{1}w_{ik})^{2}-w^{11}\nabla_{11}\chi+\frac{d}{h}+2l\frac{\sum_{j} h_{j}h_{jt}}{h-h_{t}}+\frac{1}{h-h_{t}}(1-d)\\
&\leq\sum_{i} w^{ii}d\left(\frac{h_{ii}+h-h}{h}-\frac{h^{2}_{i}}{h^{2}} \right)-2l\sum_{i} w^{ii}h^{2}_{ii}+2l\sum_{j} h_{j}\left[-\sum_{i} w^{ii}h_{jii}+\frac{h_{jt}}{h-h_{t}}\right]\\
&\quad
-w^{11}\nabla_{11}\chi+\frac{d}{h}+\frac{1}{h-h_{t}}(1-d)\\
&\leq -d\sum_{i} w^{ii}-2l\sum_{i} w^{ii}(w^{2}_{ii}-2w_{ii}h)-w^{11}\nabla_{11}\chi+\frac{nd}{h}+\frac{1}{h-h_{t}}(1-d)\\
&\quad+2l\sum_{j} h_{j}\left[\frac{h_{j}}{h-h_{t}}+\sum_{i} w^{ii}h_{j}-\nabla_{j}\chi\right]\\
&\leq-w^{11}\nabla_{11}\chi-2l\sum_{j} h_{j}\nabla_{j}\chi +(2l|\nabla h|^{2}-d)\sum_{i} w^{ii}-2l\sum_{i} w_{ii}\\
&\quad +\frac{2l|\nabla h|^{2}+1-d}{h-h_{t}}+4(n-1)hl+\frac{nd}{h}.
\end{split}
\end{equation}
On the other hand, differentiating \eqref{gaslw5}, by the first equality of \eqref{phph}, at $x_{0}$, we obtain
\begin{equation}\label{chi}
\nabla_{j}\chi=\frac{f_{j}}{f}+n\frac{hh_{j}+h_{j}h_{jj}}{\rho^{2}}+\frac{\varphi^{'}(1/h)h_{j}}{h^{2}\varphi(1/h)},
\end{equation}
and
\begin{equation}\label{chi2}
\begin{split}
\nabla_{11}\chi&=\frac{ff_{11}-f^{2}_{1}}{f^{2}}+n\frac{hh_{11}+h^{2}_{1}+h^{2}_{11}+\sum h_{j}h_{j11}}{\rho^{2}}-2n\frac{(hh_{1}+h_{1}h_{11})^{2}}{\rho^{4}}\\
&\quad -2\frac{\varphi^{'}(1/h)h^{2}_{1}}{h^{3}\varphi(1/h)}-\frac{\varphi^{''}(1/h)h^{2}_{1}}{h^{4}\varphi(1/h)}+\frac{\varphi^{'}(1/h)h_{11}}{h^{2}\varphi(1/h)}+\frac{[\varphi^{'}(1/h)]^{2}h^{2}_{1}}{h^{4}\varphi(1/h)^{2}},
\end{split}
\end{equation}
where $\varphi^{''}(1/h)$ denotes $\frac{\partial^{2}\varphi(1/h)}{\partial(1/h)^{2}}$. Using $C^{0}$, $C^{1}$ estimates in Section 4,  \eqref{gaslw1}, \eqref{chi} and \eqref{chi2}, we get
\begin{equation}
\begin{split}
\label{gaslw9*}
&-2l\sum_{j} h_{j}\nabla_{j}\chi-w^{11}\nabla_{11}\chi\\
&=-2l\sum_{j} h_{j}\left[\frac{f_{j}}{f}+n\frac{hh_{j}+h_{j}h_{jj}}{\rho^{2}}+\frac{\varphi^{'}(1/h)h_{j}}{h^{2}\varphi(1/h)}\right]\\
&\quad -w^{11}\left[\frac{ff_{11}-f^{2}_{1}}{f^{2}}+n\frac{hh_{11}+h^{2}_{1}+h^{2}_{11}+\sum h_{j}h_{j11}}{\rho^{2}}-2n\frac{(hh_{1}+h_{1}h_{11})^{2}}{\rho^{4}}\right.\\
&\left.\quad \quad  \quad \quad -2\frac{\varphi^{'}(1/h)h^{2}_{1}}{h^{3}\varphi(1/h)}-\frac{\varphi^{''}(1/h)h^{2}_{1}}{h^{4}\varphi(1/h)}+\frac{\varphi^{'}(1/h)h_{11}}{h^{2}\varphi(1/h)}+\frac{[\varphi^{'}(1/h)]^{2}h^{2}_{1}}{h^{4}\varphi(1/h)^{2}}\right]\\
& \leq C_{1}l+n\sum_{j} \frac{h_{j}}{\rho^{2}}(-2lh_{j}h_{jj}-w^{11}h_{j11})+C_{2}w^{11}+4nw^{11}\frac{hh^{2}_{1}h_{11}}{\rho^{4}}\\
&\quad +w^{11}\left[ n\frac{hh_{11}+h^{2}_{11}}{\rho^{2}}+\frac{2nh^{2}_{1}h^{2}_{11}}{\rho^{4}}+\frac{\varphi^{'}(1/h)h_{11}}{h^{2}\varphi(1/h)}\right]\\
&\leq C_{1}l+n\sum_{j} \frac{h_{j}}{\rho^{2}}\left(w^{11}h_{1}\delta_{j1}-d\frac{h_{j}}{h}\right)+C_{2}w^{11}+4nw^{11}\frac{hh^{2}_{1}(w_{11}-h)}{\rho^{4}}\\
&\quad + w^{11}\left[ n \frac{h(w_{11}-h)+(w_{11}-h)^{2}}{\rho^{2}}+\frac{2nh^{2}_{1}(w_{11}-h)^{2}}{\rho^{4}} + \frac{\varphi^{'}(1/h)(w_{11}-h)}{h^{2}\varphi(1/h)}\right],
\end{split}
\end{equation}
where $C_{1}$ is a positive constant depending on $||f||_{C^{1}(\sn)}$, $||h||_{C^{1}(\sn \times (0,\infty))}$, $||\varphi||_{C^{1}(I_{(0,\infty)})}$, $\min_{(\sn \times (0,\infty))} \rho$, $\min_{(\sn \times (0,\infty))} h$, $\min_{\sn} f$ and $\min_{I_{(0,\infty)}} \varphi$, and $C_{2}$ is a positive constant depending on $||f||_{C^{2}(\sn)}$, $||h||_{C^{1}(\sn \times (0,\infty))}$, $||\varphi||_{C^{2}(I_{(0,\infty)})}$, $\min_{(\sn \times (0,\infty))} \rho$, $\min_{\sn} f$, $\min_{I_{(0,\infty)}} \varphi$ and $\min_{(\sn \times (0,\infty))} h$. To proceed further, \eqref{gaslw9*} reduces to
\begin{equation}
\begin{split}
\label{gaslw9}
&-2l\sum_{j}h_{j}\nabla_{j}\chi-w^{11}\nabla_{11}\chi\\
&\leq C_{3}l+C_{4}w^{11}+C_{5}+n\frac{\rho^{2}+2h^{2}_{1}}{\rho^{4}}w_{11},
\end{split}
\end{equation}
where $C_{3}$, $C_{4}$, $C_{5}$ are positive constants, depending only on  $||f||_{C^{1}(\sn)}$, $||h||_{C^{1}(\sn \times (0,\infty))}$, $||\varphi||_{C^{2}(I_{(0,\infty)})}$, $\min_{(\sn \times (0,\infty))} \rho$, $\min_{(\sn \times (0,\infty))} h$, $\min_{\sn} f$ and $\min_{I_{(0,\infty)}} \varphi$.

 We substitute \eqref{gaslw9} into \eqref{gaslw8}, and choose $d=2l\max_{\sn\times (0,\infty)}|\nabla h|^{2}+1$, with
 \[
 l=n\frac{\max_{\sn \times (0,\infty)}\rho^{2}+2\max_{\sn \times (0,\infty)}|\nabla h|^{2}}{\min_{\sn \times (0,\infty)}\rho^{4}}+1.
 \]
Then, at $x_{0}$, we have
\begin{equation}
\begin{split}
\label{gaslw10}
\frac{\partial_{t}\widetilde{E}}{h-h_{t}}&\leq C_{3}l+C_{4}w^{11}+C_{5}-l\sum_{i} w_{ii}+4(n-1)hl+\frac{nd}{h}<0
\end{split}
\end{equation}
provided $w_{11}>>1$. Hence, \eqref{gaslw10} tells
\begin{equation}\label{EQULL2}
E(x_{0},t)=\widetilde{E}(x_{0},t)\leq C
\end{equation}
for some $C>0$, independent of $t$. The inequality in \eqref{EQULL2} implies that the principal radii is bounded from above. So, we complete the proof of Theorem \ref{PCUL*}.

\end{proof}

\section{Existence of smooth solution to this problem}
\label{Sec6}
In this section, we first focus on obtaining the long time existence and convergence results of flow \eqref{xOrflow}, which completes the proof of Theorem \ref{main*}.

Making use of the uniform estimates in Section 4 and Section 5, this allows us to assert that equation \eqref{xOrflow} is uniformly parabolic in $C^{2}$ norm space. Then, by virtue of the standard Krylov's regularity theory \cite{K87}, we demonstrate the long-time existence and regularity of the solution of equation\eqref{xOrflow}. Furthermore,
\begin{equation}\label{ESM1}
||h||_{C^{i,j}_{x,t}(\sn\times[0,+\infty))}\leq C
\end{equation}
for some $C>0$, independent of $t$, and for each pairs of nonnegative integers $i$ and $j$.

 With the aid of the Arzel\`a-Ascoli theorem and a diagonal argument, there exists a sequence of $t$, denoted by $\{t_{k}\}_{k\in N}\subset (0,\infty)$, and a smooth function $h(x)$ such that
\begin{equation*}
||h(x,t_{k})-h(x)||_{C^{i}({\sn})}\rightarrow 0
\end{equation*}
uniformly for any nonnegative integer $i$ as $t_{k}\rightarrow \infty$. This illustrates that $h(x)$ is a support function. Let $\Omega$ be the convex body determined by $h(x)$.  We conclude that $\Omega$ is smooth, origin symmetric and strictly convex.

We are now in position to prove the solvability of the following equation
\begin{equation*}
\gamma \varphi(1/h)h(|\nabla h|^{2}+h^{2})^{-\frac{n}{2}}det(\nabla^{2}h+hI)=f
\end{equation*}
for some positive constant $\gamma$.

On the one hand, if the condition $(i)$ holds, Lemma \ref{non-decre} tells that, for any $t>0$, we have
\begin{equation}\label{der1}
P^{'}(t)\geq 0.
\end{equation}
 In view of the $C^{0}$ estimate on $h$ showed in Lemma \ref{C0},  for any $t\geq 0$, there exists a positive constant $C$ which is independent of $t$, such that
\begin{equation}\label{der2}
P(t)\leq C.
\end{equation}
Combining \eqref{der1} and \eqref{der2}, we obtain
\begin{equation*}
\int_{0}^{t}P^{'}(s)d s=P(t)-P(0)\leq P(t)\leq C,
\end{equation*}
which implies
\begin{equation*}
\int_{0}^{\infty}P^{'}(t)dt\leq C.
\end{equation*}
This reveals that there exists a sequence of $t_{k}\rightarrow \infty$ such that
\begin{equation}\label{pe}
P^{'}(t_{k})\rightarrow 0 \ as \ t_{k}\rightarrow \infty.
\end{equation}
On the other hand, recall \eqref{derpt}, it yields
\begin{equation}
\begin{split}
\label{limt}
&P^{'}(t_{k})\int_{\sn}\frac{f}{\varphi(1/h))}dx\Bigg|_{t=t_{k}}\\
&=-\left(\int_{\sn}\sqrt{\frac{f^{2}\rho^{n}\kappa}{h\varphi(1/h)^{2}}}\sqrt{\frac{h}{\kappa \rho^{n}}}dx\right)^{2}\Bigg|_{t=t_{k}}+\left(\int_{\sn}\frac{f^{2}\rho^{n}\kappa}{h \varphi(1/h)^{2}}dx\int_{\sn}\frac{h}{\kappa \rho^{n}}dx\right)\Bigg|_{t=t_{k}}.
\end{split}
\end{equation}
Taking the limit in \eqref{limt}, using \eqref{pe}, we have
\begin{equation}\label{pcon}
\left(\int_{\sn}\sqrt{\frac{f^{2}\rho^{n}\kappa}{h\varphi(1/h)^{2}}}\sqrt{\frac{h}{\kappa \rho^{n}}}dx\right)^{2}=\int_{\sn}\frac{f^{2}\rho^{n}\kappa}{h \varphi(1/h)^{2}}dx\int_{\sn}\frac{h}{\kappa \rho^{n}}dx.
\end{equation}
In light of \eqref{derpt} and \eqref{pcon}, by the equality condition of the H\"{o}lder inequality, and using the a priori estimates in Section \ref{Sec4} and Section \ref{Sec5}, we conclude that there exists a positive constant $\gamma$ such that
\[
\frac{f^{2}\rho^{n}\kappa}{h\varphi(1/h)^{2}}=\gamma^{2} \frac{h}{\kappa \rho^{n}},
\]
which implies that
\[
\gamma \varphi(1/h)h (|\nabla h|^{2}+h^{2})^{-\frac{n}{2}}det(\nabla^{2}h+hI)=f
\]
has a solution.

On the other hand, if the condition $(ii)$ holds, Lemma \ref{non-incre} tells that, for any $t>0$,
\begin{equation*}
P^{'}(t)\leq 0,
\end{equation*}
and
\begin{equation*}
\int_{0}^{t}(-P^{'}(s))ds=P(0)-P(t)\leq P(0).
\end{equation*}
This gives
\begin{equation}\label{ptoi}
\int_{0}^{\infty}(-P^{'}(t))dt\leq P(0).
\end{equation}
Equation \eqref{ptoi} tells us that there exists a sequence of $t_{k}\rightarrow \infty$ such that
\begin{equation*}
-P^{'}(t_{k})\rightarrow 0 \ as \ t_{k}\rightarrow \infty.
\end{equation*}
Following the similar line as above, we conclude that
\[
\gamma \varphi(1/h)h (|\nabla h|^{2}+h^{2})^{-\frac{n}{2}}det(\nabla^{2}h+hI)=f
\]
has a solution for some positive constant $\gamma$. Hence, we complete the proof of Theorem \ref{main*}.

As presented in the introduction, for  general $\varphi$, there may be no uniqueness of solutions to  equation \eqref{Or-Mong}. Here, following the similar lines as in \cite{CW06,LY20}, we shall give a special uniqueness result of solutions to equation \eqref{Or-Mong} based on mild monotonicity assumption on $\varphi$ in the case of $\gamma=1$ . It is shown in the following.
\begin{theo}\label{unique}
Under the assumptions of Theorem \ref{main*}, assume moreover that $\varphi$ is increasing on $(0,\infty)$. Then the solutions to equation
\begin{equation}\label{uneq}
 \varphi(1/h)h (|\nabla h|^{2}+h^{2})^{-\frac{n}{2}}det(\nabla^{2}h+hI)=f
\end{equation}
is unique.
\end{theo}
\begin{proof}
Assume $h_{1}$ and $h_{2}$ be two solutions of equation \eqref{uneq}. To prove $h_{1}=h_{2}$, on the one  hand, we take by contradiction and assume that $\max_{\sn} \frac{h_{1}}{h_{2}}>1$. Suppose $\frac{h_{1}}{h_{2}}$ achieves its maximum at $x_{0}\in\sn$. It follows $h_{1}(x_{0})>h_{2}(x_{0})$. Let $G={\rm log}\frac{h_{1}}{h_{2}}$. So, at $x_{0}$, one has
\begin{equation}\label{fide}
0=\nabla G=\frac{\nabla h_{1}}{h_{1}}-\frac{\nabla h_{2}}{h_{2}},
\end{equation}
and applying \eqref{fide}, we deduce
\begin{equation}
\begin{split}
\label{sede}
0&\geq \nabla^{2}G\\
&=\frac{\nabla^{2}h_{1}}{h_{1}}-\frac{\nabla h_{1}\otimes \nabla h_{1} }{h^{2}_{1}}-\frac{\nabla^{2}h_{2}}{h_{2}}+\frac{\nabla h_{2}\otimes \nabla h_{2} }{h^{2}_{2}}\\
&=\frac{\nabla^{2}h_{1}}{h_{1}}-\frac{\nabla^{2}h_{2}}{h_{2}}.
\end{split}
\end{equation}
Since $h_{1}$ and $h_{2}$ are solutions of equation \eqref{uneq}, using \eqref{uneq} and \eqref{sede}, at $x_{0}$,  we obtain
\begin{equation}
\begin{split}
\label{h1h2}
1&=\frac{\varphi(1/h_{2})h_{2}(|\nabla h_{2}|^{2}+h^{2}_{2})^{-\frac{n}{2}}det(\nabla^{2}h_{2}+h_{2}I)}{\varphi(1/h_{1})h_{1}(|\nabla h_{1}|^{2}+h^{2}_{1})^{-\frac{n}{2}}det(\nabla^{2}h_{1}+h_{1}I)}\\
&=\frac{\varphi(1/h_{2})h_{2}\left[h^{-n}_{2}\left(|\frac{\nabla h_{2}}{h_{2}}|^{2}+1\right)^{-\frac{n}{2}}\right]h^{n-1}_{2}det\left(\frac{\nabla^{2}h_{2}}{h_{2}}+I\right)}{\varphi(1/h_{1})h_{1}\left[h^{-n}_{1}\left(|\frac{\nabla h_{1}}{h_{1}}|^{2}+1\right)^{-\frac{n}{2}}\right]h^{n-1}_{1}det\left(\frac{\nabla^{2}h_{1}}{h_{1}}+I\right)}\\
&\geq \frac{\varphi(1/h_{2})}{\varphi(1/h_{1})}.
\end{split}
\end{equation}
In light of the assumption in Theorem \ref{unique}, together with \eqref{h1h2}, it implies that $h_{1}(x_{0})\leq h_{2}(x_{0})$, which is a contradiction. This reveals
\begin{equation}\label{h1xiao}
\max_{\sn} \frac{h_{1}}{h_{2}}\leq 1.
\end{equation}
On the other hand, interchanging the role of $h_{1}$ and $h_{2}$, applying the same argument as above, we have
\begin{equation}\label{h2xiao}
\max_{\sn}\frac{h_{2}}{h_{1}}\leq 1.
\end{equation}
Combining \eqref{h1xiao} and \eqref{h2xiao}, this illustrates that $h_{1}=h_{2}$. So, we complete the proof.

\end{proof}

\section*{Acknowledgment}The authors are grateful to their supervisor Prof. Yong Huang for his constant guidance and encouragement.

\end{document}